\documentclass[12pt, reqno]{amsart}

\usepackage[top=3.65cm, bottom=3cm, left=3cm, right=3cm]{geometry}
\usepackage{mathtools}
\usepackage{scrextend}
\usepackage{eucal}
\usepackage{amsmath}
\usepackage{amsthm}
\usepackage{amsfonts}
\usepackage{amssymb}
\usepackage{bm}
\usepackage[all,cmtip]{xy}
\usepackage{hyperref}
\usepackage{color}

\usepackage[scr=boondoxo]{mathalfa}

\numberwithin{equation}{section}

\theoremstyle{plain}
\newtheorem{theorem}{Theorem}[section]
\newtheorem*{question}{Question}
\newtheorem{lemma}[theorem]{Lemma}
\newtheorem{proposition}[theorem]{Proposition}

\theoremstyle{definition}
\newtheorem{definition}[theorem]{Definition}
\newtheorem{remark}[theorem]{Remark}
\newtheorem{example}[theorem]{Example}

\newcommand{\Db}{\mathrm{D^b}}

\DeclareMathOperator{\Proj}{Proj}
\DeclareMathOperator{\Sym}{Sym}
\DeclareMathOperator{\Spec}{Spec}

\DeclareMathOperator{\Hom}{Hom}
\DeclareMathOperator{\Tor}{Tor}

\newcommand{\image}{\mathrm{Im}}
\newcommand{\pr}{\mathrm{pr}}
\newcommand{\id}{\mathrm{id}}
\DeclareMathOperator{\Cone}{Cone}
\DeclareMathOperator{\Pic}{Pic}
\newcommand{\Cliff}{\mathscr{C}\!\ell}

\newcommand{\inv}{\mathrm{inv}}

\newcommand{\Br}{\mathrm{Br}}

\newcommand{\cO}{\mathcal{O}}
\newcommand{\cA}{\mathcal{A}}

\newcommand{\cC}{\mathcal{C}}

\newcommand{\cE}{\mathcal{E}}
\newcommand{\cF}{\mathcal{F}}
\newcommand{\cG}{\mathcal{G}}

\newcommand{\cL}{\mathcal{L}}

\newcommand{\cT}{\mathcal{T}}

\newcommand{\cX}{\mathcal{X}}

\newcommand{\rL}{\mathrm{L}}
\newcommand{\rR}{\mathrm{R}}
\newcommand{\rT}{\mathrm{T}}

\newcommand{\bP}{\mathbf{P}}
\newcommand{\bA}{\mathbf{A}}
\newcommand{\bF}{\mathbf{F}}
\newcommand{\bR}{\mathbf{R}}

\newcommand{\bZ}{\mathbf{Z}}
\newcommand{\bC}{\mathbf{C}}
\newcommand{\bQ}{\mathbf{Q}}
\newcommand{\bG}{\mathbf{G}}

\newcommand{\pA}{\CMcal{A}}

\newcommand{\pF}{\CMcal{F}}

\begin{document}

\title[Rational points on twisted K3 surfaces and derived equivalences]{Rational points on twisted K3 surfaces \\[2pt] and derived equivalences}
\author{Kenneth Ascher}
\email{kenascher@math.brown.edu}
\address{Mathematics Department, Brown University}
\author{Krishna Dasaratha}
\email{kdasarat@math.stanford.edu}
\address{Department of Mathematics, Stanford University}
\author{Alexander Perry}
\email{aperry@math.harvard.edu}
\address{Department of Mathematics, Harvard University}
\author{Rong Zhou}
\email{rzhou@math.harvard.edu}
\address{Department of Mathematics, Harvard University}

\thanks{This project was supported by the 2015 AWS and NSF grant DMS-1161523.  K.A. was also partially supported by NSF grant DMS-1162367.
A.P. was also partially supported by NSF GRFP grant DGE1144152.}

\begin{abstract} 
Using a construction of Hassett and V\'arilly-Alvarado, we produce derived 
equivalent twisted K3 surfaces over $\bQ$, $\bQ_2$, and $\bR$, where one has a rational point 
and the other does not. 
This answers negatively a question recently raised by Hassett and Tschinkel. 
\end{abstract}

\maketitle

\section{Introduction}
A twisted K3 surface is a pair $(X, \alpha)$, where $X$ is a K3 surface and $\alpha \in \Br(X)$ is a Brauer class.
In a recent survey paper~\cite{ht}, Hassett and Tschinkel asked whether the existence of a 
rational point on a twisted K3 surface is invariant under derived equivalence. 
More precisely, they asked: 

\begin{question}\label{conj}
Let $(X_1, \alpha_1)$ and $(X_2, \alpha_2)$ be twisted K3 surfaces over a field $k$. 
Suppose there is a $k$-linear equivalence 
\begin{equation*}
\Db(X_1, \alpha_1) \simeq \Db(X_2, \alpha_2) 
\end{equation*} 
of twisted derived categories. 
Then is the existence of a $k$-point of $(X_1, \alpha_1)$ equivalent to the existence 
of a $k$-point of $(X_2, \alpha_2)$?
\end{question}

By definition, a $k$-point of a twisted K3 surface $(X, \alpha)$ is a point $x \in X(k)$ such that 
the evaluation $\alpha(x) = 0 \in \Br(k)$. 
Equivalently, it is a $k$-point of the $\bG_m$-gerbe over $X$ associated to $\alpha$. 

In~\cite{ht}, it is shown that for the untwisted case of the question where $\alpha_1, \alpha_2$ vanish, 
the answer is positive over certain fields $k$, 
e.g. $\bR$, finite fields, and $p$-adic fields 
(provided the $X_i$ have good reduction, or $p \geq 7$ and the $X_i$ have ADE reduction).
The purpose of this paper is to show that if $\alpha_1, \alpha_2$ are allowed to be nontrivial, 
the answer to the question is negative for $k = \bQ, \bQ_2$, and $\bR$. 

We work over a field $k$ of characteristic not equal to $2$, 
and consider a double cover $Y \to \bP^2 \times \bP^2$ ramified over a divisor of bidegree $(2,2)$.
The projection $\pi_i: Y \to \bP^2$ onto the $i$-th $\bP^2$ factor, $i=1,2$, realizes $Y$ as a quadric 
fibration. Provided that the discriminant divisor of $\pi_i$ is smooth, the Stein factorization of the relative Fano 
variety of lines of $\pi_i$ is a K3 surface $X_i$, which comes with a natural Brauer class $\alpha_i$. 
In this setup, we prove the following result. 

\begin{theorem}
\label{theorem-equivalence-intro}
There is a $k$-linear equivalence $\Db(X_1, \alpha_1) \simeq \Db(X_2, \alpha_2)$.
\end{theorem}

We note that this result seems to be known to the experts 
(at least for $k = \bC$), but we could not find a proof in the literature. 

Hassett and V\'arilly-Alvarado studied the above construction of twisted K3s in relation to rational points \cite{hva}.
They show that over $k = \bQ$, 
if certain conditions are imposed on the branch divisor $Z \subset \bP^2 \times \bP^2$ of $Y$, 
the class $\alpha_1$ gives a (transcendental) Brauer--Manin obstruction to the Hasse principle on $X_1$. 
A priori, $\alpha_2$ need not obstruct the existence of rational points on $X_2$. 
In fact, it is possible that $X_2$ has rational points, 
but the conditions imposed on $Z$ result in very large coefficients of the defining equation of $X_2$, 
making a computer search for points infeasible.

In this paper, we observe that the $2$-adic condition imposed by Hassett and V\'arilly-Alvarado can be relaxed, 
while still guaranteeing $\alpha_1$ gives a Brauer--Manin obstruction (see Lemma \ref{2-adic}). The upshot is that the defining coefficients 
of $X_2$ are much smaller, making it easy to find rational points with a computer. Up to modifying the $\alpha_i$ by 
a Brauer class pulled back from $k = \bQ$, we obtain the desired example over $\bQ$. 
We also check the example ``localizes'' over $\bQ_2$ and $\bR$. More precisely, we prove: 

\begin{theorem}
\label{theorem-counterexample-intro}
For $k = \bQ, \bQ_2$, or $\bR$, the divisor $Z\subset \bP^2 \times \bP^2$ can be chosen so that there are Brauer classes 
$\alpha'_i \in \Br(X_i)$, congruent to $\alpha_i$ modulo $\image(\Br(k) \to \Br(X_i))$, such that:
\begin{enumerate}
\item There is a $k$-linear equivalence $\Db(X_1, \alpha'_1) \simeq \Db(X_2, \alpha'_2)$,
\item $(X_1, \alpha'_1)$ has no $k$-point,
\item $(X_2, \alpha'_2)$ has a $k$-point.
\end{enumerate}
\end{theorem}

\subsection*{Acknowledgements} 
The authors thank the organizers of the 2015 Arizona Winter School, where this work began. 
We thank Asher Auel, Brendan Hassett, and Sho Tanimoto for useful discussions. 
Above all, we thank Tony V\'arilly-Alvarado for suggesting this project, and for his help and encouragement along the way. 


\section{Construction of the twisted K3 surfaces}
\label{section-construction}

In this section, $k$ denotes a base field of characteristic not equal to $2$.

\subsection{Quadric fibrations} 
\label{subsection-quadric-fibrations}
We start by reviewing some terminology on quadric fibrations. 
Let $S$ be a variety over $k$, i.e. an integral, separated scheme of finite type over $k$. 
Let $\cE$ be a rank $n \geq 2$ vector bundle on $S$, i.e. a locally free $\cO_S$-module of rank $n$. 
Our convention is that the projective bundle of $\cE$ is the morphism
\begin{equation*}
p: \bP(\cE) = \Proj_S(\Sym^{\bullet}(\cE^*)) \to S.
\end{equation*}
A quadric fibration is determined by a line bundle $\cL$ on $S$ and a nonzero section 
\begin{equation*}
s \in \Gamma(\bP(\cE), \cO_{\bP(\cE)}(2) \otimes p^*\cL) = \Gamma(S, \Sym^2(\cE^*) \otimes \cL).
\end{equation*}
Namely, the zero locus of $s$ on $\bP(\cE)$ defines a subvariety $Q$, and the restriction $\pi: Q \to S$ 
of $p: \bP(\cE) \to S$ is the associated \emph{quadric fibration}, which if flat is of relative dimension $n-2$. 
Below we will be specifically interested in flat quadric fibrations of relative dimension $2$, 
which we refer to as \emph{quadric surface fibrations}.

Note that the section of $\Sym^2(\cE^*) \otimes \cL$ defining a quadric fibration
corresponds to a morphism $q: \cE \to \cE^{*} \otimes \cL$. 
Taking the determinant gives rise to a section of $\det(\cE^*)^2 \otimes \cL^n$
whose vanishing defines the \emph{discriminant locus} $D \subset S$, which is a divisor 
provided $\pi: Q \to S$ is generically smooth. 
The fibration $\pi: Q \to S$ is said to have \emph{simple degeneration} if the fiber over every closed 
point of $S$ is a quadric of corank $\leq 1$. 
We note that if $\pi: Q \to S$ is flat and generically smooth and $S$ is smooth over $k$, 
then the discriminant divisor $D$ is smooth over $k$ if and only if $Q$ is smooth over $k$ and 
$\pi$ has simple degeneration~\cite[Proposition 1.6]{quadric-fibrations-auel}.

\subsection{Twisted K3 surfaces} 
\label{subsection-twisted-K3-construction}

Let $V_1$ and $V_2$ be $3$-dimensional vector spaces over $k$. 
We denote by $H_i$ the hyperplane class on $\bP(V_i)$; by abuse of notation, we  
denote by the same letter the pullback of $H_i$ to any variety mapping to $\bP(V_i)$. 
Let 
\begin{equation*}
\pi: Y \to \bP(V_1) \times \bP(V_2) 
\end{equation*}
be the double cover of $\bP(V_1) \times \bP(V_2)$ 
ramified over a smooth divisor $Z$ in the linear system $|2H_1 + 2H_2|$. 
Let $\pr_i :  \bP(V_1) \times \bP(V_2) \to \bP(V_i)$ be the $i$-th projection, 
and let $\pi_i = \pr_i \circ \pi: Y \to \bP(V_i)$.

\begin{lemma}
\label{lemma-quadric-fibration}
Let $\cE_1 = (V_2 \otimes \cO) \oplus \cO(H_1)$ on $\bP(V_1)$ and 
$\cE_2 = (V_1 \otimes \cO) \oplus \cO(H_2)$ on $\bP(V_2)$. 
Then for $i=1,2$ there is a commutative diagram
\begin{equation*}
\xymatrix{
Y \ar[d]_{\pi_i} \ar[r]^{j_i} & \bP(\cE_i) \ar[dl]^{p_i} \\
\bP(V_i)
}
\end{equation*}
where $j_i$ is a closed immersion with $j_1^*\cO_{\bP(\cE_1)}(1) = \cO_Y(H_2)$ 
and $j_2^*\cO_{\bP(\cE_2)}(1) = \cO_Y(H_1)$.
Moreover, $Y$ is cut out on $\bP(\cE_i)$ by a section of $\cO_{\bP(\cE_i)}(2) \otimes \cO(2H_i)$, 
so that $\pi_i$ is a quadric surface fibration.
\end{lemma}

\begin{proof}
Consider the case $i = 1$. 
The morphism $j_1 : Y \to \bP(\cE_1)$ is given by the $\pi_1$-very ample line bundle $\cO_Y(H_2)$. 
More precisely, using 
$\pi_*(\cO_Y) = \cO \oplus \cO(-H_1-H_2)$, 
we find 
\begin{align*}
\pi_{1*}(\cO_Y(H_2)) & = \pr_{1*}(\cO(H_2) \oplus \cO(-H_1)) \\
& = (V_2^* \otimes \cO) \oplus \cO(-H_1) \\
& = \cE_1^*.
\end{align*}
Working locally on $\bP(V_1)$, we see the canonical map 
$\pi_1^*\cE_1^* = \pi_1^*\pi_{1*}(\cO_Y(H_2)) \to \cO_Y(H_2)$ is 
surjective and the corresponding morphism $j_1: Y \to \bP(\cE_1)$ is an immersion. 
By construction $j_1^*\cO_{\bP(\cE_1)}(1) = \cO_Y(H_2)$. 
Moreover, if $\zeta$ denotes the class of $\cO_{\bP(\cE_1)}(1)$ in $\Pic(\bP(\cE_1))$, 
then it is easy to compute 
\begin{equation*}
[Y] = 2\zeta + 2H_1 \in \Pic(\bP(\cE_1))
\end{equation*}
by using the intersection numbers $H_1^2H_2^2 = 2$ and $H_1H_2^3 = 0$ on $Y$.
So $Y$ is indeed a quadric surface fibration, cut out by a section of $\cO_{\bP(\cE_1)}(2) \otimes \cO(2H_1)$ 
on $\bP(\cE_1)$. 
\end{proof}

Let $D_i$ denote the discriminant divisor of $\pi_i: Y \to \bP(V_i)$. 
It follows from the lemma that $D_i$ is defined by a section of 
$\det(\cE_i^*)^2 \otimes \cO(8H_i) = \cO(6H_i)$, i.e. $D_i \subset \bP(V_i)$ is a sextic curve. 
Let $f_i : X_i \to \bP(V_i)$ be the double cover of $\bP(V_i)$ ramified over $D_i$. 
If $D_i$ is smooth (equivalently, if $\pi_i$ has simple degeneration), then 
$X_i$ is a smooth K3 surface. 
Moreover, $X_i$ comes equipped with an Azumaya algebra $\pA_i$, as follows. 

In general, consider a generically smooth quadric surface fibration $\pi: Q \to S$ over a smooth 
$k$-variety $S$, with smooth discriminant divisor and simple degeneration. 
Let $\pF \to S$ be the relative Fano variety of lines of $\pi$. 
It follows from~\cite[Proposition~3.3]{transcendental-obstructions} 
that Stein factorization gives morphisms
\begin{equation*}
\pF \xrightarrow{\, g \,} X \xrightarrow{\, f \,} S,
\end{equation*} 
where $g$ is an \'{e}tale locally trivial $\bP^1$-bundle over $X$ 
and $f$ is the double cover of $S$ branched along the discriminant divisor $D$. 
The morphism $g$ corresponds to an Azumaya algebra $\pA$ on $X$. 

Applying this discussion to $\pi_i: Y \to \bP(V_i)$, we see that if $D_i$ is smooth, 
then $X_i$ is equipped with an Azumaya algebra $\pA_i$. 
Of course $\pA_i$ represents a Brauer class $\alpha_i \in \Br(X_i)$, so we can regard the pair 
$(X_i, \pA_i)$ as a twisted K3 surface.


\section{Derived equivalence of the twisted K3 surfaces}
\label{section-equivalence}

In this section, we prove the twisted K3 surfaces $(X_i, \pA_i)$ of the 
previous section are derived equivalent. 
Our proof works over any field $k$ of characteristic not equal to~$2$, 
and gives an explicit functor inducing the equivalence.
The key tool is Kuznetsov's semiorthogonal decomposition of the 
derived category of a quadric fibration~\cite{quadric-fibrations-kuznetsov}. 

\subsection{Conventions}
\label{section-conventions}
All triangulated categories appearing below will be $k$-linear, and functors between 
them will be $k$-linear and exact. 

For a variety $X$, 
we denote by $\Db(X)$ the bounded derived category of coherent sheaves on $X$, 
regarded as a triangulated category.  
More generally, for any sheaf of $\cO_X$-algebras $\pA$ 
which is coherent as an $\cO_X$-module, we denote by $\Db(X, \pA)$ the 
bounded derived category of coherent sheaves of right $\pA$-modules on $X$. 
We note that if $\pA$ is an Azumaya algebra corresponding to a Brauer class $\alpha \in \Br(X)$, 
then the bounded derived category of $\alpha$-twisted sheaves $\Db(X, \alpha)$ is equivalent 
to $\Db(X, \pA)$. 

As a rule, all functors we consider are derived. 
More precisely, for a morphism of varieties $f: X \to Y$, we simply write $f_*: \Db(X) \to \Db(Y)$ 
for the derived pushforward (provided $f$ is proper) and $f^*: \Db(Y) \to \Db(X)$ 
for the derived pullback (provided $f$ has finite $\Tor$-dimension).  
Similarly, for $\cF, \cG \in \Db(X)$, we write $\cF \otimes \cG \in \Db(X)$ for the derived 
tensor product. 

\subsection{Semiorthogonal decompositions}
\label{section-sod}
One way to understand the derived category of a variety (or more generally a triangulated 
category) is by ``decomposing'' it into simpler pieces. This is formalized by the notion of a 
semiorthogonal decomposition, which plays a central role in the rest of this section. 
We summarize the rudiments of this theory; see e.g.~\cite{bondal} and \cite{bondal-kapranov} 
for a more detailed exposition. 

\begin{definition}
Let $\cT$ be a triangulated category. 
A \emph{semiorthogonal decomposition} 
\begin{equation*}
\cT = \langle \cA_1, \dots, \cA_n \rangle
\end{equation*}
is a sequence of full triangulated subcategories $\cA_1, \dots, \cA_n$ of $\cT$ 
--- called the \emph{components} of the decomposition --- such that: 
\begin{enumerate}
\item $\Hom(\cF, \cG) = 0$ for all $\cF \in \cA_i, \cG \in \cA_j$ if $i>j$.
\item For any $\cF \in \cT$, there is a sequence of morphisms
\begin{equation*}
0 = \cF_n \to \cF_{n-1} \to \cdots \to \cF_1 \to \cF_0 = \cF,
\end{equation*}
such that $\Cone(\cF_i \to \cF_{i-1}) \in \cA_i$.
\end{enumerate}
\end{definition}

Semiorthogonal decompositions are closely related to the notion of 
an \emph{admissible subcategory} of a triangulated category. 
Such a subcategory $\cA \subset \cT$ is by definition a 
full triangulated subcategory such that the inclusion $i: \cA \hookrightarrow \cT$ 
admits right and left adjoints $i^!: \cT \to \cA$ and $i^*: \cT \to \cA$. 
For $X$ a smooth proper variety over $k$, the components 
of any semiorthogonal decomposition of $\Db(X)$ are in fact admissible subcategories. 

The simplest examples of admissible subcategories come from exceptional objects. 
An object $\cF \in \cT$ of a triangulated category is called \emph{exceptional} if 
\begin{equation*}
\Hom(\cF, \cF[p]) = 
\left \{ \begin{array}{ll}
k & \mbox{if } p = 0, \\
0 & \mbox{if } p \neq 0.
\end{array}
\right.
\end{equation*}
If $X$ is a proper variety and $\cF \in \Db(X)$ is exceptional, then the full triangulated 
subcategory $\langle \cF \rangle \subset \Db(X)$ generated by $\cF$ is admissible and equivalent 
to the derived category of a point via 
$\Db(\Spec(k)) \to \Db(X): V \mapsto V \otimes \cF$.  
To simplify notation, we write $\cF$ in place of $\langle \cF \rangle$ when 
$\langle \cF \rangle$ appears as a component in a semiorthogonal decomposition, 
i.e.\ instead of $\Db(X) = \langle \dots, \langle \cF \rangle, \dots \rangle$ we write 
$\Db(X) = \langle \dots, \cF , \dots \rangle$.

\begin{example}
\label{example-beilinson}
It is easy to see any line bundle on projective space $\bP^n$ 
is exceptional as an object of $\Db(\bP^n)$. 
In fact, Beilinson~\cite{beilinson} showed $\Db(\bP^n)$ has a semiorthogonal decomposition 
into $n+1$ line bundles, namely
\begin{equation*}
\Db(\bP^n) = \langle \cO, \cO(1), \dots, \cO(n) \rangle. 
\end{equation*}
\end{example}

Given one semiorthogonal decomposition of a triangulated category $\cT$, others 
can be obtained via mutation functors. 
If $i: \cA \hookrightarrow \cT$ is the inclusion of an admissible subcategory, 
the \emph{left} and \emph{right mutation functors} $\rL_{\cA}: \cT \to \cT$ 
and $\rR_{\cA}: \cT \to \cT$ are 
defined by the formulas 
\begin{equation*}
\rL_{\cA}(\cF) = \Cone(ii^!\cF \to \cF) 
\quad \text{and} \quad 
\rR_{\cA}(\cF) = \Cone(\cF \to ii^*\cF)[-1], 
\end{equation*}
where $ii^!\cF \to \cF$ and $\cF \to ii^*\cF$ are the counit and unit morphisms of the 
adjunctions. 
These functors satisfy the following basic properties.  

\begin{lemma}
\label{lemma-mutation-functor-properties}
The mutation functors $\rL_{\cA}$ and $\rR_{\cA}$ annihilate $\cA$. 
Moreover, they restrict to mutually inverse equivalences  
\begin{equation*}
\rL_{\cA}|_{{}^{\perp}\cA} : {}^{\perp}\cA \xrightarrow{\sim} \cA^{\perp} 
\quad \text{and} \quad
\rR_{\cA}|_{\cA^{\perp}}  : \cA^{\perp} \xrightarrow{\sim} {}^{\perp}\cA, 
\end{equation*}
where $\cA^{\perp}$ and ${}^{\perp}\cA$ are the \emph{right} and \emph{left orthogonal} categories 
to $\cA$, i.e. the full subcategories of $\cT$ defined by 
\begin{align*}
\cA^{\perp} & = \left \{ \cF \in \cT ~ | ~ \Hom(\cG, \cF) = 0 \text{ for all } \cG \in \cA \right \},  \\
{}^{\perp}\cA & =  \left \{ \cF \in \cT ~ | ~ \Hom(\cF, \cG) = 0 \text{ for all } \cG \in \cA \right \}.
\end{align*} 
\end{lemma}

The following lemma describes the action of mutation functors on a 
semiorthogonal decomposition. 
\begin{lemma}
\label{lemma-mutations}
Let $\cT = \langle \cA_1,\dots,\cA_n \rangle$ be a semiorthogonal decomposition 
with admissible components. 
Then for $1 \leq i \leq n-1$ there is a semiorthogonal decomposition
\begin{equation*}
\cT  = \langle \cA_1, \dots, \cA_{i-1}, \rL_{\cA_i}(\cA_{i+1}), \cA_i, \cA_{i+2}, \dots, \cA_n \rangle,
\end{equation*}
and for $2 \leq i \leq n$ there is a semiorthogonal decomposition
\begin{equation*}
\cT  = \langle \cA_1, \dots, \cA_{i-2}, \cA_i, \rR_{\cA_i}(\cA_{i-1}), \cA_{i+1}, \dots, \cA_n \rangle.
\end{equation*}
\end{lemma}

We will also need the following lemma, which allows us to compute the effect of 
a mutation functor in a special case. It follows easily from Serre duality.

\begin{lemma}
\label{lemma-mutation-serre-functor}
Let $X$ be a smooth projective variety over $k$, and let 
$\Db(X) = \langle \cA_1, \dots, \cA_n \rangle$ be a semiorthogonal decomposition.  
Then $\rL_{\langle \cA_1, \dots, \cA_{n-1}\rangle}(\cA_n) = \cA_n \otimes \omega_X$, where 
$\cA_n \otimes \omega_X$ denotes the image of $\cA_n$ under the autoequivalence 
$\cF \mapsto \cF \otimes \omega_X$ of $\Db(X)$. 
\end{lemma}


\subsection{Derived categories of quadric fibrations} 
\label{section-quadric-fibrations-sod}
Let $\pi: Q \to S$ be a quadric fibration associated to a rank $n$ vector bundle $\cE$ and a section 
of $\Sym^2(\cE^*) \otimes \cL$, as in Section~\ref{subsection-quadric-fibrations}. 
Then there is an associated \emph{even Clifford algebra} $\Cliff_0$,
which is a sheaf of algebras on $S$ given as a certain 
quotient of the tensor algebra $\rT^{\bullet}(\cE \otimes \cE \otimes \cL^*)$. 
For the precise definition, see~\cite[Section 1.5]{quadric-fibrations-auel} 
(cf.~\cite[Section 3.3]{quadric-fibrations-kuznetsov}). 
We note that $\Cliff_0$ admits an $\cO_S$-module 
filtration of length $\lfloor \frac{n}{2} \rfloor$ with 
associated graded pieces $\wedge^{2i} \cE \otimes (\cL^*)^i$.

In case the fibration $\pi: Q \to S$ is flat and $S$ is smooth over $k$, 
Kuznestov~\cite{quadric-fibrations-kuznetsov} established a semiorthogonal 
decomposition of $\Db(Q)$ into a copy of $\Db(S, \Cliff_0)$ and a number of copies of $\Db(S)$. 
In fact, Kuznetsov stated his result under the assumption that $k$ is algebraically closed of characteristic $0$, 
but as explained in \cite[Theorem~2.11]{quadric-fibrations-auel}, the proof works without this hypothesis.

\begin{theorem}[{\cite[Theorem 4.2]{quadric-fibrations-kuznetsov}}]
\label{theorem-quadric-fibration-sod}
Let $\pi: Q \to S$ be a flat quadric fibration of relative dimension $n-2$ over a smooth $k$-variety $S$. 
Let $\cO_{Q}(1)$ denote the restriction of $\cO_{\bP(\cE)}(1)$ to $Q$. 
Then the functor $\pi^*: \Db(S) \to \Db(Q)$ is fully faithful, and there is a fully faithful functor 
$\Phi: \Db(S, \Cliff_0) \to \Db(Q)$ such that there is a semiorthogonal decomposition 
\begin{equation*}
\Db(Q) = \langle \Phi (\Db(S, \Cliff_0)), \pi^* \Db(S) \otimes \cO_{Q}(1), \dots, 
\pi^* \Db(S) \otimes \cO_{Q}(n-2) \rangle.
\end{equation*}
\end{theorem}

\begin{remark}
The functor $\Phi: \Db(S, \Cliff_0) \to \Db(Q)$ is given by an explicit Fourier--Mukai kernel, 
see~\cite[Section 4]{quadric-fibrations-kuznetsov}.
\end{remark}

Now assume $\pi: Q \to S$ is a generically smooth quadric surface fibration 
over a smooth $k$-variety $S$, with smooth discriminant divisor and simple degeneration.  
As in the discussion at the end of Section~\ref{subsection-twisted-K3-construction}, 
the double cover $f: X \to S$ ramified over $D$ is equipped with an Azumaya 
algebra $\pA$. In terms of this data, we have the following alternative description 
of $\Db(S, \Cliff_0)$, see~\cite[Proposition B.3]{quadric-fibrations-auel} 
or~\cite[Lemma~4.2]{cubics-kuznetsov}.

\begin{lemma}
\label{lemma-cliff-surface-fibration}
In the above situation, there is an isomorphism $f_*\pA \cong \Cliff_0$. 
In particular, pushforward by $f$ induces an equivalence 
$f_{*}: \Db(X, \pA) \xrightarrow{\sim} \Db(S, \Cliff_0)$.
\end{lemma}

\subsection{Derived equivalence}
\label{section-derived-equivalence}

Let $\pi: Y \to \bP(V_1) \times \bP(V_2)$ be as in Section~\ref{subsection-twisted-K3-construction}. 
Assume the discriminant divisors $D_i$ of the quadric fibrations $\pi_i: Y \to \bP(V_i)$ are smooth, 
so that we get associated twisted K3 surfaces $(X_i, \pA_i)$. 
Let $\Cliff_{0,i}$ denote the even Clifford algebra of the quadric fibration ${\pi_i: Y \to \bP(V_i)}$. 
Then Lemma~\ref{lemma-cliff-surface-fibration} gives an equivalence ${f_{i*}: \Db(X_i, \pA_i) \xrightarrow{\sim} \Db(\bP(V_i), \Cliff_{0,i})}$. 
Finally, let ${\Phi_i: \Db(\bP(V_i), \Cliff_{0,i}) \to \Db(Y)}$ be the fully faithful functor from 
Theorem~\ref{theorem-quadric-fibration-sod}. 
In this setup, we prove the following result.

\begin{theorem}
\label{theorem-derived-equivalence}
Assume $D_1$ and $D_2$ are smooth. 
Then there is an equivalence 
\begin{equation*}
\Db(X_1, \pA_1) \simeq \Db(X_2, \pA_2)
\end{equation*}
given by the composition
\begin{equation*}
\label{equivalence-cliff}
f_{2*}^{-1} \circ \Phi_2^* \circ \rR_{\cO_Y(H_2)} \circ \rL_{\cO_Y(H_1)} \circ \Phi_1 \circ f_{1*}: 
\Db(X_1, \pA_1) \to \Db(X_2, \pA_2),
\end{equation*}
where 
\begin{itemize}
\item $\rL_{\cO_Y(H_1)}$ is the left mutation functor through $\langle \cO_Y(H_1) \rangle \subset \Db(Y)$, 
\item $\rR_{\cO_Y(H_2)}$ is the right mutation functor through $\langle \cO_Y(H_2) \rangle \subset \Db(Y)$, 
\item $\Phi_2^*$ is the left adjoint of $\Phi_2$, 
\item $f_{2*}^{-1}$ is the inverse of the equivalence 
$f_{2*}: \Db(X_2, \pA_2) \xrightarrow{\sim} \Db(\bP(V_2), \Cliff_{0,2})$.
\end{itemize}
\end{theorem}

The theorem is an immediate consequence of the following proposition. 
We note that the proposition holds without assuming smoothness of the 
discriminant divisors $D_i$. 

\begin{proposition}
\label{proposition-cliff-equivalence}
There is an equivalence 
\begin{equation*}
\Db(\bP(V_1), \Cliff_{0,1}) \simeq \Db(\bP(V_2), \Cliff_{0,2})
\end{equation*}
given by the composition 
\begin{equation*}
\Phi_2^* \circ \rR_{\cO_Y(H_2)} \circ \rL_{\cO_Y(H_1)} \circ \Phi_1: 
\Db(\bP(V_1), \Cliff_{0,1}) \to \Db(\bP(V_2), \Cliff_{0,2}).
\end{equation*}
\end{proposition}

\begin{proof}
Set $\cC_i = \Phi_i(\Db(\bP(V_i), \Cliff_{0,i})) \subset \Db(Y)$.
Theorem~\ref{theorem-quadric-fibration-sod} gives semiorthogonal decompositions
\begin{align*}
\Db(Y) & = \langle \cC_1, \pi_1^*\Db(\bP(V_1)) \otimes \cO(H_2), \pi_1^*\Db(\bP(V_1)) \otimes \cO(2H_2) \rangle, \\
\Db(Y) & = \langle \cC_2, \pi_2^*\Db(\bP(V_2)) \otimes \cO(H_1), \pi_2^*\Db(\bP(V_2)) \otimes \cO (2H_1) \rangle.
\end{align*}
Recall Beilinson's decomposition $\Db(\bP(V_i)) = \langle \cO, \cO(H_i), \cO(2H_i) \rangle$ 
(see Example~\ref{example-beilinson}). 
In each of the above decompositions of $\Db(Y)$, we replace the first copy of 
$\Db(\bP(V_i))$ by Beilinson's decomposition and the second copy by the same decomposition 
twisted by $\cO(H_i)$:
\begin{align}
\label{proof-sod-1}
\phantom{\Db(Y)}
& \begin{aligned}
\mathllap{\Db(Y)} =  \langle \cC_1, \, & \cO(H_2), \cO(H_1+H_2), \cO(2H_1+H_2), \\
& \cO(H_1+2H_2), \cO(2H_1+2H_2), \cO(3H_1+2H_2) \rangle,
\end{aligned} \\
\label{proof-sod-2}
& \begin{aligned}
\mathllap{\Db(Y)} = \langle \cC_2, \, & \cO(H_1), \cO(H_1+H_2), \cO(H_1+2H_2) , \\
& \cO(2H_1+H_2), \cO(2H_1+2H_2), \cO(2H_1+3H_2) \rangle.
\end{aligned}
\end{align}
We perform a sequence of mutations that identifies the categories 
generated by the exceptional objects in~\eqref{proof-sod-1} and~\eqref{proof-sod-2}.

First consider~\eqref{proof-sod-1}. 
Mutate $\cO(3H_1+2H_2)$ to the far left of the decomposition. 
Note that $Y$ is smooth with canonical class $K_Y = -2H_1-2H_2$, 
so by Lemma~\ref{lemma-mutation-serre-functor} the result of the mutation is
\begin{align*}
\Db(Y)  = \langle \cO(H_1), \cC_1,  \, & \cO(H_2), \cO(H_1+H_2), \cO(2H_1+H_2), \\
& \cO(H_1+2H_2), \cO(2H_1+2H_2) \rangle.
\end{align*}
Left mutating $\cC_1$ through $\cO(H_1)$ then gives a decomposition
\begin{equation}
\begin{split}
\label{proof-sod-1-new}
\Db(Y)  = \langle \rL_{\cO(H_1)}\cC_1, \cO(H_1), \, & \cO(H_2), \cO(H_1+H_2), \cO(2H_1+H_2), \\
& \cO(H_1+2H_2), \cO(2H_1+2H_2) \rangle.
\end{split}
\end{equation}
By the same argument, we obtain from~\eqref{proof-sod-2} a similar decomposition
\begin{equation}
\label{proof-sod-2-new}
\begin{split}
\Db(Y) = \langle \rL_{\cO(H_2)}\cC_2, \cO(H_2), \, & \cO(H_1), \cO(H_1+H_2), \cO(H_1+2H_2), \\
& \cO(2H_1+H_2), \cO(2H_1+2H_2) \rangle. 
\end{split}
\end{equation}
Up to permutation, the exceptional objects in the decompositions~\eqref{proof-sod-1-new} and 
~\eqref{proof-sod-2-new} agree, hence they generate the same subcategory of $\Db(Y)$. 
It follows that $\rL_{\cO(H_1)}\cC_1$ and $\rL_{\cO(H_2)}\cC_2$ coincide, 
as both are the right orthogonal to the same subcategory.
Now the proposition follows since 
$\rR_{\cO(H_2)} \circ \rL_{\cO(H_2)} \cong \id$ 
on ${^{\perp}}\langle \cO(H_2) \rangle$ 
by Lemma~\ref{lemma-mutation-functor-properties}.
\end{proof}

\begin{remark}
\label{remark-picard-number}
The equivalence $\Db(X_1, \alpha_1) \simeq \Db(X_2, \alpha_2)$ of 
Theorem~\ref{theorem-derived-equivalence} implies other relations between $X_1$ and $X_2$. 
For instance, over $k = \bC$ it implies the Picard numbers of $X_1$ and $X_2$ agree. 
Indeed, it suffices to note that the equivalence induces an isomorphism of twisted 
transcendental lattices $\rT(X_1, \alpha_1) \cong \rT(X_2, \alpha_2)$, whose 
ranks are the same as the usual transcendental lattices (see~\cite{huybrechts-stellari}).
\end{remark}


\section{Equations for the twisted K3 surfaces and local invariants}
\label{section-local-invariants}
Let $k$ be a number field.
Then for any place $v$ of $k$, class field theory provides an 
embedding $\inv_v: \Br(k_v) \to \bQ/\bZ$ (which is an isomorphism for nonarchimedian $v$). 
Now let $X$ be a smooth, projective, geometrically integral variety over $k$. 
Any subset $S \subset \Br(X)$ cuts out a subset $X(\bA_k)^S \subset X(\bA_{k})$ of the adelic points of $X$, given by
\begin{equation*}
X(\bA_{k})^{S} = \left \{ (x_v) \in X(\bA_k) ~ | ~ \sum_{v} \inv_v \, \alpha(x_v) = 0  \text{ for all } \alpha \in S \right\}.
\end{equation*}
For fixed $(x_v) \in X(\bA_k)$ and $\alpha \in \Br(X)$, the evaluation $\alpha(x_v) = 0$ for all but finitely many $v$, so the above sum is well-defined. 
Moreover, class field theory gives inclusions 
\begin{equation*}
X(k) \subset X(\bA_k)^S \subset X(\bA_k).
\end{equation*}
Hence, if $X(\bA_k)^S$ is empty for some $S$, then $X$ has no $k$-points. 
We note that if $X(\bA_k)^S$ is empty but $X(\bA_k)$ is not, then $S$ is said to give a 
\emph{Brauer--Manin obstruction} to the Hasse principle. 
See~\cite[5.2]{Sko} for more details.

In this section, we describe conditions on the $(2,2)$ divisor $Z \subset \bP(V_1) \times \bP(V_2)$ 
from Section~\ref{subsection-twisted-K3-construction}, which allow us to control the 
local invariants $\inv_{v} \, \alpha_1(x_v)$ for any $v$-adic point $x_v \in X_1(k_v)$. 
In the end, we will see that if $k = \bQ$ and enough conditions are met, then 
\begin{equation*}
\inv_v \, \alpha_1(x_v) = 
\begin{cases}
0 & \text{if $v$ is finite}, \\
\frac{1}{2} & \text{if $v$ is real},
\end{cases}
\end{equation*}
for all $x_v \in X_1(k_v)$. 
Hence $X_1(\bA_k)^{\alpha_1}$ is empty and $X_1$ has no $k$-points. 
Our discussion follows~\cite{hva} very closely, and differs only in the treatment of the $2$-adic place 
(see Lemma~\ref{2-adic}). 

\subsection{Equations for the twisted K3 surfaces} 
Let the notation be as in Section~\ref{subsection-twisted-K3-construction} 
(in particular $k$ may be any field of characteristic not equal to $2$).

Choose coordinates $x_0, x_1, x_2$ on $\bP(V_1)$ and $y_0,y_1,y_2$ on $\bP(V_2)$. 
The equation defining $Z$ can be written as 
\begin{equation}
\label{A-F}
\begin{aligned}
& A(x_0, x_1, x_2)y_0^2 + B(x_0, x_1, x_2)y_0y_1 + C(x_0,x_1,x_2)y_0y_2 \, + \\
& D(x_0, x_1, x_2)y_1^2 + E(x_0, x_1, x_2)y_1y_2 + F(x_0,x_1,x_2)y_2^2, 
\end{aligned} 
\end{equation}
where $A, \dots, F$ are degree $2$ homogeneous polynomials in the $x_i$, 
or as 
\begin{equation}
\label{A-F-prime}
\begin{aligned}
& A'(y_0, y_1, y_2)x_0^2 + B'(y_0, y_1, y_2)x_0x_1 + C'(y_0,y_1,y_2)x_0x_2 \, + \\
& D'(y_0, y_1, y_2)x_1^2 + E'(y_0, y_1, y_2)x_1x_2 + F'(y_0,y_1,y_2)x_2^2.
\end{aligned}
\end{equation}
where $A', \dots, F'$ are degree $2$ homogeneous polynomials in the $y_i$. 
The first or second expression is useful depending on whether we regard $Y$ 
as a quadric fibration over $\bP(V_1)$ or $\bP(V_2)$. 
The following lemma summarizes the computations of~\cite[Section 3]{hva}.

\begin{lemma}
\label{lemma-equation}
$(1)$ Let
\begin{equation*}
M  = 
\begin{pmatrix}
2A & B & C \\
B & 2D & E \\
C & E & 2F 
\end{pmatrix} 
\end{equation*} 
Then the discriminant curve $D_1 \subset \bP(V_1)$ is defined by 
$\det(M) = 0$, 
and $X_1$ is defined  
in the weighted projective space $\bP(1,1,1,3)$ with coordinates $x_0,x_1,x_2,w$ by
\begin{equation*}
w^2 = -\frac{1}{2} \det(M).
\end{equation*}
The analogous statements hold for $D_2 \subset \bP(V_2)$ and $X_2$ with $M$ 
replaced by 
\begin{equation*}
M'  = 
\begin{pmatrix}
2A' & B' & C' \\
B' & 2D' & E' \\
C' & E' & 2F'
\end{pmatrix}.
\end{equation*}

\medskip \noindent
$(2)$ Define 
\begin{equation*}
M_A = 4DF - E^2, \quad M_D = 4AF - C^2, \quad M_F = 4AD - B^2.
\end{equation*}
Assume $D_1 \subset \bP(V_1)$ is smooth, so that we have a twisted K3 surface 
$(X_1, \alpha_1)$. 
Then the image of $\alpha_1$ under the injection $\Br(X_1) \to \Br(k(X_1))$  
(where $k(X_1)$ is the function field of $X_1$) can be represented by any of the following 
Hilbert symbols: 
\begin{equation*}
(-M_F, A), ~ (-M_D, A), ~ (-M_F, D), ~ (-M_A, D), ~ (-M_D, F), ~ (-M_A, F).
\end{equation*}
Defining $M'_{A'}, M'_{D'}, M'_{F'}$ similarly, the analogous statement holds for $\alpha_2 \in \Br(X_2)$. 
\end{lemma}

From now on, we assume $D_1 \subset \bP(V_1)$ is smooth, so that $(X_1, \alpha_1)$ is defined.

\subsection{Conditions controlling the local invariants}
\label{subsection-conditions}
The following result holds by Proposition 4.1 and Lemma 4.2 of~\cite{hva}. 
It allows us to control local invariants at finite places of bad reduction, assuming the place is not $2$-adic 
and the singularities are mild.

\begin{proposition}
\label{badred}
Let $F$ be a finite extension of $\bQ_p$ for $p \neq 2$, and denote by $\cO_F$ the 
ring of integers of $F$. 
Let $X$ be a K3 surface over $F$. 
Let $\cX \to \Spec(\cO_F)$ be a flat, proper morphism from a regular scheme $\cX$, 
with generic fiber $\cX_{\eta} \cong X$. 
Assume the singular locus of the geometric special fiber $\cX_{\overline{s}}$ 
consists of less than $8$ points, each of which is an ordinary double point. 
If $X(F) \neq \emptyset$, then for any $2$-power torsion Brauer class 
$\alpha \in \Br(X)[2^{\infty}]$, the map $X(F) \to \Br(F)$ given by evaluation of $\alpha$ 
is constant. 
In particular, $\alpha(x) = 0$ for all $x \in X(F)$ if this holds for a single $x$. 
\end{proposition}

The next result is~\cite[Lemma 4.4]{hva}. It guarantees that the local invariants of 
$\alpha_1 \in \Br(X_1)$ vanish at finite places of good reduction, away from the prime $2$.  
\begin{lemma}
\label{goodred}
Let $k$ be a number field.
Let $v$ be a finite place of good reduction for $X_1$ which is not $2$-adic. 
Then $\inv_{v} \, \alpha_1(x) = 0$ for all $x \in X_1(k_v)$. 
\end{lemma}

We are left to control the real and $2$-adic invariants of $\alpha_1 \in \Br(X_1)$. 
The following result, which is~\cite[Corollary 4.6]{hva}, 
gives conditions which guarantee $\alpha_1$ is nontrivial at any real point of $X_1$.
\begin{lemma}
\label{real}
Let $k = \bQ$. Assume the polynomials $A, \dots, F$ from~\eqref{A-F}, 
when regarded as quadratic forms, satisfy: 
\begin{enumerate}
\item $A,D,$ and $F$ are negative definite,
\item $B,C,$ and $E$ are positive definite. 
\end{enumerate}
If $\infty$ denotes the real place, then $\inv_{\infty} \, \alpha_1(x) = 1/2$ for all $x \in X_1(\bR)$. 
\end{lemma}

The following lemma improves~\cite[Lemma~4.7]{hva}, giving conditions such that  
$\alpha_1$ is trivial at every $2$-adic point of $X_1$. 

\begin{lemma}
\label{2-adic}
Let $k = \bQ$. Write the polynomials $A, \dots, F \in \bQ[x_0,x_1,x_2]$ from~\eqref{A-F} as
\begin{align*} 
A  & = A_1 x_0^2 + A_2 x_0x_1 + A_3 x_0x_2 + A_4 x_1^2 + A_5x_1x_2 + A_6 x_2^2,
\\ B  & = B_1 x_0^2 + B_2 x_0x_1 + B_3 x_0x_2 + B_4 x_1^2 + B_5x_1x_2 + B_6 x_2^2,
\\ C  & = C_1 x_0^2 + C_2 x_0x_1 + C_3 x_0x_2 + C_4 x_1^2 + C_5x_1x_2 + C_6 x_2^2,
\\ D  & = D_1 x_0^2 + D_2 x_0x_1 + D_3 x_0x_2 + D_4 x_1^2 + D_5x_1x_2 + D_6 x_2^2,
\\ E  & = E_1 x_0^2 + E_2 x_0x_1 + E_3 x_0x_2 + E_4 x_1^2 + E_5x_1x_2 + E_6 x_2^2, 
\\ F  & = F_1 x_0^2 + F_2 x_0x_1 +F_3 x_0x_2 + F_4 x_1^2 + F_5x_1x_2 + F_6 x_2^2.
\end{align*} 
Suppose the coefficients of $A, \dots, F$ satisfy:
\begin{enumerate}
\item The $2$-adic valuation of $A_1,B_1,C_6, D_4, E_4,$ and $F_6$ is $0$.  
\item The $2$-adic valuation of all other coefficients is $>0$.
\end{enumerate} 
Then $\inv_2 \, \alpha_1(x) = 0$ for all $x \in X_1(\bQ_2)$.  
 \end{lemma}
 
\begin{proof}
Let $x = [x_0, x_1, x_2, w]$ be a point of  $X_1(\bQ_2) \subset \bP(1,1,1,3)(\bQ_2)$. 
By scaling the coordinates, we may assume $x_0, x_1, x_2 \in \bZ_2$ and at least 
one of the $x_i$ is a unit. 
By Lemma~\ref{lemma-equation}, the Hilbert symbols
\begin{equation*}
(B^2-4AD, A), \quad (E^2-4DF, D), \quad (C^2-4AF, F)
\end{equation*}
all represent the image of $\alpha_1$ in $\Br(k(X))$. 
According to whether $x_0,x_1,$ or $x_2$ is a $2$-adic unit, the 
first, second, or third of these representatives can be used to see 
$\inv_2 \, \alpha_1(x) = 0$. 

For instance, suppose $x_0$ is a $2$-adic unit.  
Then by our assumptions on coefficients,
\begin{equation*}
A(x) \quad \text{and} \quad B(x)^2-4A(x)D(x)
\end{equation*} 
are also 2-adic units. 
In particular, they are nonzero, so 
\begin{equation*}
(B(x)^2-4A(x)D(x), A(x))_2
\end{equation*}
represents $\alpha_1(x) \in \Br(\bQ_2)$. 
Recall (see for example \cite[p.\ 20, Theorem 1]{serre}) that if $s,t \in \bZ_2^{\times}$, then 
\begin{equation*}
(s,t)_2 = (-1)^{\frac{s-1}{2}\frac{t-1}{2}}.
\end{equation*}
But by our assumptions 
\begin{equation*}
B(x)^2-4A(x)D(x) \equiv 1 \ (\text{mod}\ 4), 
\end{equation*}
so the formula gives
\begin{equation*}
(B(x)^2-4A(x)D(x), A(x))_2 = 1.
\end{equation*}
Thus $\alpha_1(x) = 0 \in \Br(\bQ_2)$. 

The same argument works when $x_1$ or $x_2$ is a $2$-adic unit, using the other 
representatives for $\alpha_1$ from above. 
\end{proof}


\section{Proof of Theorem~\ref{theorem-counterexample-intro}}
\label{section-proof-counterexample}

Consider the following quadrics in $\bZ[x_0,x_1,x_2]$:
\begin{align*} 
 A  & = -5 x_0^2  + 4 x_0x_2 -4 x_1^2 + 2x_1x_2 -4 x_2^2,
\\ B  & = 5x_0^2 +  2x_0x_1 -2 x_0x_2 + 2x_1^2 + 2x_1x_2 +  4x_2^2,
\\ C  & =  4x_0^2 + 2x_0x_1 -4x_0x_2 +2x_1^2 -2x_1x_2 +  5x_2^2,
\\ D  & =  -4x_0^2  -2x_0x_1  - x_1^2 -2x_1x_2 -4  x_2^2,
\\ E  & =  4x_0^2  +3 x_1^2  +  4x_2^2, 
\\ F  & =  -4x_0^2 + 4x_0x_1 +2x_0x_2 -2 x_1^2 -4x_1x_2 -5 x_2^2.
\end{align*}
Inserting these polynomials in~\eqref{A-F} gives the equation of a (2,2) divisor 
\begin{equation*}
Z \subset \bP(V_1) \times \bP(V_2),
\end{equation*}
which we regard as a variety over $\bQ$. 
As in Section~\ref{subsection-twisted-K3-construction}, $Z$ gives rise to a branched double 
cover $\pi: Y \to \bP(V_1) \times \bP(V_2)$, which is a quadric fibration via projection to each factor. 
Lemma~\ref{lemma-equation} gives explicit equations for the discriminant curves $D_i \subset \bP(V_i)$, 
and the Jacobian criterion can be used to check the $D_i$ are smooth. 
Hence, by Section~\ref{subsection-twisted-K3-construction}, we get associated twisted K3 surfaces $(X_i, \alpha_i)$, 
which we will use to prove Theorem~\ref{theorem-counterexample-intro}.

\begin{remark}
The quadrics $A, \dots, F$ above were found using the algorithm described 
in~\cite[Section 6]{hva}, modified in two ways. First, we omitted the steps related to checking 
the geometric Picard number of $X_1$ is $1$, 
since it was not our goal to produce an example with this property. 
Second, instead of using~\cite[Lemma~4.7]{hva} to constrain the quadrics, 
we used our Lemma~\ref{2-adic}, which results in much smaller coefficients. 
Indeed, the equations for $X_1$ and $X_2$ are: 
\begin{align*}
 & \begin{aligned}  
 w^2 = & -4x_0^6 - 308x_0^5x_1 - 190x_0^4x_1^2 - 278x_0^3x_1^3 - 203x_0^2x_1^4 - 40x_0x_1^5 - 28x_1^6\\&
    + 18x_0^5x_2 + 460x_0^4x_1x_2 + 276x_0^3x_1^2x_2 + 474x_0^2x_1^3x_2 + 40x_0x_1^4x_2 \\&+
    98x_1^5x_2 - 25x_0^4x_2^2 - 820x_0^3x_1x_2^2 - 247x_0^2x_1^2x_2^2 - 374x_0x_1^3x_2^2\\& -
    2x_1^4x_2^2 + 20x_0^3x_2^3 + 652x_0^2x_1x_2^3 + 14x_0x_1^2x_2^3 + 270x_1^3x_2^3 \\&-
    20x_0^2x_2^4 - 562x_0x_1x_2^4 - 105x_1^2x_2^4 - 8x_0x_2^5 + 166x_1x_2^5 - 4x_2^6 ,
    \end{aligned}   \\ 
& \begin{aligned} 
   w^2 = & ~ 236y_0^6 - 740y_0^5y + 1268y_0^4y_1^2 - 1092y_0^3y_1^3 + 624y_0^2y_1^4 - 164y_0y_1^5\\ &+
    32y_1^6 - 616y_0^5y_2+ 416y_0^4y_1y_2 - 96y_0^3y_1^2y_2 - 976y_0^2y_1^3y_2 +
    548y_0y_1^4y_2 \\& - 288y_1^5y_2 + 1236y_0^4y_2^2 - 456y_0^3y_1y_2^2 + 1484y_0^2y_1^2y_2^2-
    356y_0y_1^3y_2^2  \\ & + 676y_1^4y_2^2 - 1332y_0^3y_2^3 - 804y_0^2y_1y_2^3 - 372y_0y_1^2y_2^3
 -1024y_1^3y_2^3 + 1036y_0^2y_2^4     \\& + 768y_0y_1y_2^4 + 812y_1^2y_2^4 - 472y_0y_2^5 -
    388y_1y_2^5 + 40y_2^6.
    \end{aligned} 
\end{align*}
In contrast, most of the coefficients appearing in the corresponding equations in \cite{hva} have 5 or 6 digits.
Smaller coefficients are crucial in making a computer search for points of $X_2$ feasible.
\end{remark}

The following proposition is reduced to a series of computations by the 
results in Section~\ref{subsection-conditions}. 
We postpone its proof to the end of the section. 

\begin{proposition}
\label{proposition-no-point}
$(1)$ $X_1(\bQ_v) \neq \emptyset$ for all places $v$, or equivalently $X_1(\bA_\bQ)\neq\emptyset$.

\medskip \noindent
$(2)$ 
The local invariants of the class $\alpha_1 \in \Br(X_1)$ satisfy
\begin{equation*}
\inv_v \, \alpha_1(x_v) = 
\begin{cases}
0 & \text{if $v$ is finite}, \\
\frac{1}{2} & \text{if $v$ is real},
\end{cases}
\end{equation*}
for all $x_v \in X_1(\bQ_v)$. 
In particular, $\alpha_1$ obstructs the existence of $\bQ$-points on $X_1$, 
and hence gives a Brauer--Manin obstruction to the Hasse principle. 
\end{proposition}

Using the proposition, we construct the examples which prove Theorem~\ref{theorem-counterexample-intro}. 
The necessary computations that appear below were carried out using Magma\footnote{Code available at \url{http://www.math.brown.edu/\~kenascher/magma/magma.html}.}.

\subsection{Example over $\bQ$}
Using the equation for $X_2$ given by Lemma~\ref{lemma-equation}, 
it can be checked that $x = [1,1,1,0] \in \bP(1,1,1,3)$
is a $\bQ$-point of $X_2$. 
Let $\beta = \alpha_2(x) \in \Br(\bQ)$, 
let $\beta_i$ be the constant class given by the image of 
$\beta$ under $\Br(\bQ) \to \Br(X_i)$, 
and set $\alpha'_i = \beta_i^{-1}\alpha_i$. 

Then there is a $\bQ$-linear equivalence $\Db(X_1, \alpha'_1) \simeq \Db(X_2, \alpha'_2)$ 
induced by the equivalence $\Db(X_1, \alpha_1) \simeq \Db(X_2, \alpha_2)$ of 
Theorem~\ref{theorem-derived-equivalence}.
By Proposition~\ref{proposition-no-point}, $X_1$ has no $\bQ$-point, so a fortiori  
the pair $(X_1, \alpha'_1)$ has no $\bQ$-point. 
On the other hand, by construction $x$ is a $\bQ$-point of $(X_2, \alpha'_2)$. 

\subsection{Example over $\bQ_2$}
Replace the pairs $(X_i, \alpha_i)$ defined above over $\bQ$ by their base changes 
to $\bQ_2$. It can be checked that $x = [-3,-1,1,\sqrt{357008}] \in \bP(1,1,1,3)$ is a $\bQ_2$-point of 
$X_2$ (note that Hensel's lemma can be used to see $357008$ is a $2$-adic square).
One then checks that $\beta = \alpha_2(x) = (B(x)^2-4A(x)D(x),A(x))_{2}$ is nontrivial. 
Let $\beta_i$ be the constant class given by the image of 
$\beta$ under $\Br(\bQ) \to \Br(X_i)$, 
and set $\alpha'_i = \beta_i^{-1}\alpha_i$. 

Then there is a $\bQ_2$-linear equivalence $\Db(X_1, \alpha'_1) \simeq \Db(X_2, \alpha'_2)$ induced 
by the equivalence of Theorem~\ref{theorem-derived-equivalence}.
By Proposition~\ref{proposition-no-point}, $\alpha_1(y)$ is trivial for any $y \in X_1(\bQ_2)$, 
and hence $\alpha'_1(y) = \alpha_2^{-1}(x)\alpha_1(y)$ is nontrivial (since $\alpha_2(x)$ is). 
Thus $(X_1, \alpha'_1)$ has no $\bQ_2$-points. 
On the other hand, by design $x$ is a point of $(X_2, \alpha'_2)$. 

\subsection{Example over $\bR$}
Replace the pairs $(X_i, \alpha_i)$ by their base changes to $\bR$. 
Then Theorem~\ref{theorem-derived-equivalence} still gives an $\bR$-linear equivalence 
$\Db(X_1, \alpha_1) \simeq \Db(X_2, \alpha_2)$. 
Moreover, Proposition~\ref{proposition-no-point} shows $\alpha_1(x)$ is nontrivial 
for any $x \in X_1(\bR)$, so $(X_1, \alpha_1)$ has no $\bR$-points. 
On the other hand, using Lemma~\ref{lemma-equation}, it can be checked 
that the point 
\begin{equation*}
x = [4,3,3, \sqrt{5204}] \in \bP(1,1,1,3)
\end{equation*}
lies on $X_2$ and that 
$\alpha_2(x) = (B(x)^2-4A(x)D(x),A(x))_{\infty}$ is trivial. 
Hence $x$ is an $\bR$-point of $(X_2, \alpha_2)$. 

\subsection{Proof of Proposition~\ref{proposition-no-point}}

\subsubsection{Local points}
We first check that $X_1(\bQ_v)\neq\emptyset$ for all $v$. 
This is obvious when $v=\infty$.
Let $v=p$ be a finite prime of good reduction with $p>22$. 
Then if $(X_1)_p$ is a smooth reduction of $X_1$ at $p$, there is an  
$\bF_p$-point of $(X_1)_p$ by the Weil conjectures. 
This lifts to a $\bQ_p$-point of $X_1$ by Hensel's lemma.

It therefore suffices to check that $X_1(\bQ_p) \neq \emptyset$ for primes $p$ of bad reduction for $X_1$ 
and for all primes $p<22$. 
A Gr\"obner basis calculation as in \cite[Section 5.1]{hva} can be used to show the primes of bad 
reduction for $X_1$ are:
 \begin{center} 
2,
5,
7,
307,
4591,
27077,
371857,
6902849,
104388233,\\
541264119547919951,
6097863609641310921149279, \\
2616678388926286398002864469014842817095009312844790479
\end{center} 
In the table below, we list for each prime $p$ of bad reduction and each $p<22$ 
the $(x_0,x_1,x_2)$ coordinates of a $\bQ_p$-point of $X_1$. 
(By Lemma~\ref{lemma-equation}, $(x_0,x_1,x_2)$ gives a $\bQ_p$-point if 
 $-\frac{1}{2}\det(M)(x_0,x_1,x_2)$ is a square in $\bQ_p$, which can be checked using Hensel's lemma). 

\medskip  
\begin{center}
\resizebox{15cm}{!} {\begin{tabular}{| l | c |  } \hline $p$ & $(x_0,x_1,x_2) $ \\
\hline
2&(-1,0,-1) \\ 
3&(-1,-1,1) \\
5 &(-1,-1,0) \\
7 &(-1,-1,1) \\
11 &(-1,-1,0) \\
13 &(-1,-1,1) \\
17 &(-1,-1,-1) \\
19 &(-1,-1,-1) \\
307 &(-1,-1,-1) \\
4591 &(-1,-1,0) \\
27077 &(-1,-1,-1) \\
371857 &(-1,-1,-1) \\
6902849 &(-1,0,0) \\
104388233 &(-1,-1,-1) \\
541264119547919951 & (-1,-1,1) \\
6097863609641310921149279 & (-1,1,-1) \\
2616678388926286398002864469014842817095009312844790479 & (-1,-1,0) \\
\hline
\end{tabular}}
\end{center}
\smallskip

\subsubsection{Local invariants} 
One computes that for each prime $p \neq 2$ of bad reduction, 
$X_{1, \bQ_p}$ satisfies the assumptions of Proposition~\ref{badred}. 
Moreover, using the representatives for $\alpha_1$ given in Lemma~\ref{lemma-equation}, 
it can be computed that $\alpha_1$ is trivial when evaluated at the $\bQ_p$-points specified in the table above. 
We conclude by Proposition~\ref{badred} that $\inv_v \, \alpha_1(x_v) = 0$ at the non-2-adic finite places 
$v$ of bad reduction. On the other hand, at the non-2-adic finite places of good reduction, we 
also have $\inv_v \, \alpha_1(x_v) = 0$ by Lemma~\ref{goodred}. 

Finally, it is straightforward to check the quadrics $A, \dots, F$ satisfy the hypotheses 
of Lemmas~\ref{real} and~\ref{2-adic}. The conclusions of these lemmas give  
Proposition~\ref{proposition-no-point}(2) at the real and $2$-adic places. \qed


\bibliographystyle{alpha}
\bibliography{derivedk3}

\end{document}